\documentclass[a4paper,11pt]{article}
\usepackage[latin1]{inputenc}
\usepackage[T1]{fontenc}
\usepackage{amsmath,amssymb,amstext}
\usepackage{theorem}
\usepackage{verbatim}
\usepackage{a4wide}

\newtheorem{theorem}{Theorem}[section]

\newtheorem{definition}[theorem]{Definition}
\newtheorem{example}[theorem]{Example}

\newtheorem{problem}{Problem}
\newtheorem{proposition}[theorem]{Proposition}
\newtheorem{remark}[theorem]{Remark}

\newenvironment{proof}[1][Proof]{\noindent\textbf{#1.} }{\ \rule{0.5em}{0.5em}}
\newcommand{\R}{\mathbb{R}}

\newcommand{\E}{\mathbb{E}}
\newcommand{\p}{\mathbb{P}}

\newcommand{\N}{\mathcal{N}}
\newcommand{\Nd}{\mathcal{N}_{d}}
\newcommand{\h}{\mathfrak{H}}

\newcommand{\1}{\mathbf{1}}
\newcommand{\var}{\mathrm{\bf Var}}

\begin{document}

\begin{center}
{\Large{\bf Universal Gaussian fluctuations on the discrete Poisson chaos}}
\normalsize
\\~\\ by Giovanni Peccati \footnote{Facult\'{e} des Sciences, de la Technologie et de la Communication; UR en Math\'{e}matiques, Universit\'{e} du Luxembourg . 6, rue Richard
Coudenhove-Kalergi, L-1359 Luxembourg. Email: \texttt{giovanni.peccati@gmail.com}} and Cengbo Zheng \footnote{LAMA, Universit\'{e} Paris-Est Marne-la-Vall\'{e}e, 5 Boulevard Descartes,  77454 Champs sur Marne,  and LPMA, Universit\'{e} Paris VI, Paris, France. Email: \texttt{zhengcb@gmail.com }} \\ ~\\
\end{center}

{\small \noindent {\bf Abstract}: We prove that homogenous sums inside a fixed discrete Poisson chaos are universal with respect to normal approximations. This result parallels some recent findings, in a Gaussian context, by Nourdin, Peccati and Reinert (2010). As a by-product of our analysis, we provide some refinements of the CLTs for random variables on the Poisson space proved by Peccati, Sol\'e, Taqqu and Utzet (2010), and by Peccati and Zheng (2010).}



\begin{section}{Introduction}
\subsection{Overview}

A {\it universality result} is a mathematical statement implying that the asymptotic behaviour of a large random system does not depend on the distribution of its components. Universality results are one of the leading themes of modern probability, distinguished examples being the Central Limit Theorem (CLT), the Donsker Theorem, or the Semicircular and Circular Laws in random matrix theory.

In this paper, we shall prove a new class of universality statements involving {\it homogeneous sums} based on a sequence of centered independent Poisson random variables. Homogeneous sums (see Definition \ref{d:hs}) are almost ubiquitous probabilistic objects: for instance, they provide archetypal examples of $U$-statistics, and they are the building blocks of such fundamental collections of random variables as the {\it Gaussian Wiener chaos}, the {\it Poisson Wiener chaos} or the {\it Walsh chaos}. See e.g. \cite{nourdin_peccati_book, npr4, nualart, pt_book, priv}, as well as the forthcoming Section \ref{ss:frame}, for an introduction to these concepts.

Our findings extend to the Poisson framework the results of \cite{npr2}, by Nourdin, Peccati and Reinert, where the authors discovered a remarkable universality property involving the normal approximation of homogeneous sums living inside a fixed Gaussian Wiener chaos. According to \cite{npr2}, the following universal phenomenon takes indeed place:

\begin{quote}
{\sl Let $\{F_n\}$ be a sequence of random variables such that each $F_n$ is a homogeneous sum of a fixed order $\geq 2$ based on a sequence of i.i.d. standard Gaussian random variables, and assume that $\{F_n\}$ verifies a CLT. Then, the CLT continues to hold if one replaces the i.i.d. Gaussian sequence, inside the definition of each $F_n$, with a generic collection of independent and identically distributed random variables with mean zero and unit variance. } (See Theorem \ref{t:unigauss} below for a precise statement).
\end{quote}

To describe this fact, one says that homogeneous sums inside the Gaussian Wiener chaos are {\it universal} with respect to normal approximations.

In this paper, we shall address the following natural question: {\sl are there other examples of homogeneous sums that enjoy the same universal property?} As anticipated, our proof of the universal character of homogeneous sums inside the Poisson Wiener chaos will yield a positive answer. As discussed in Remark \ref{r:cex}, and similarly to the Gaussian case, our conclusions do not extend to sums of order 1.

It is important to note that \cite{npr2} also contains an elementary counterexample, implying that homogeneous sums based on Rademacher sequences are not universal. This argument is reproduced in the proof of Proposition \ref{p:walshnot} below.

The findings of the present work are a continuation of the theory developed in \cite{pstu, peczheng}, respectively by Peccati, Sol\'e, Taqqu and Utzet and by Peccati and Zheng, where the authors combined two probabilistic techniques, namely the {\it Stein's method} for probabilistic approximations and the {\it Malliavin calculus of variations}, in order to compute explicit bounds in (possibly multidimensional) CLTs involving functionals of a given Poisson field. One of our main findings, see Theorem \ref{thm_HS_Poisson} below, provides a substantial refinement these results, which is indeed an analogous for Poisson homogeneous sums of the `fourth moment theorem' proved by Nualart and Peccati in \cite{nuapec}. One should note that the study of normal approximations on the Poisson space has recently gained much relevance, specifically in connection with stochastic geometry -- see \cite{DFR,FV,PLR1, lesmathias, SchTh}.

Other relevant references are the paper by Mossel {\it et al.} \cite{MOO}, containing an invariance principle on which \cite{npr} is based,
and de Jong \cite{deJongBook, deJongMulti}, where one can find remarkable CLTs for general degenerate $U$-statistics.

The subsequent Section \ref{ss:frame} contains a formal introduction to the objects studied in this paper. From now on, we assume that every random element is defined on a common probability space $(\Omega, \mathcal{F}, \p)$.

\subsection{Framework and motivation}\label{ss:frame}

The following three objects will play a crucial role in our discussion.
\begin{itemize}
  \item[--] ${\bf G} = \{ G_i : i\geq 1 \}$ indicates a collection of independent and identically distributed (i.i.d.) Gaussian random variables such that $G_i \sim \N(0,1)$;
  \item[--] ${\bf E} = \{ e_i : i\geq 1 \}$ denotes a {\it Rademacher sequence}, that is, the random variables $e_i$ are i.i.d. and such that  $\p(e_i = 1) = \p(e_i = -1) = \frac{1}{2} $ for every $i \geq 1$;
  \item[--] $\textbf{P} = \{ P_i : i\geq 1 \}$ stands for a collection of independent random variables such that
  \begin{equation}\label{e:pi}
  P_i \stackrel{\rm law}{  = } \frac{P(\lambda_i)-\lambda_i}{\sqrt{\lambda_i}},\quad i\geq 1,
  \end{equation}
  where $P(\lambda_i)$ indicates a Poisson random variable with parameter $\lambda_i\in (0,\infty)$.
\end{itemize}
We now formally introduce the notion of \emph{homogeneous sum}.
\begin{definition}[Homogeneous sums]\label{d:hs}
Fix some integers $1\leq q\leq N$, and write $[N]$ for the set $\{ 1,2,\cdots,N \}$. Let ${\bf X} = \{X_i : i\geq 1\}$ be a collection of independent random variables, and let $f: [N]^q \rightarrow \R$ be a {symmetric function vanishing on diagonals} (i.e. $f(i_1,\cdots,i_q)=0$ if $\exists k \neq l: i_k = i_l$). The random variable
$$ Q_q = Q_q(N,f,{\bf X}) = \sum_{1\leq i_1,\cdots,i_q \leq N} f(i_1,\cdots,i_q) X_{i_1}\cdots X_{i_q}$$
is called the {\bf multilinear homogeneous sum}, of order $q$, based on $f$ and on the first $N$ elements of ${\bf X}$. Plainly, a homogeneous sum of order $1$ is a finite sum of the type $\sum_{i=1}^Nf(i)X_i$.
\end{definition}
\begin{remark}
\rm
If, for $i=1,2,\ldots$,  $\E[X_i]=0$ and $\E[X_i^2]=1 $ (as e.g. for ${\bf X} ={\bf G},\, {\bf E}$ or ${\bf P}$), then we deduce immediately that the mean and variance of $ Q_q= Q_q(N,f,{\bf X})$ are given by:
$$ \E[Q_q]=0 , \qquad \E[Q_q^2]= q! \sum_{1\leq i_1,\cdots,i_q \leq N} f^2(i_1,\cdots,i_q).$$
\end{remark}

The next three examples show that homogeneous sums based on {\bf G}, {\bf E} and \textbf{P} can always be represented as `chaotic random variables'. The reader is referred to \cite{nourdin_peccati_book, nualart} and \cite{npr4}, respectively, for definitions and results concerning the Gaussian Wiener chaos and the Walsh chaos. An introduction to the Poisson Wiener chaos is provided in Section \ref{sec_prelim} below.

\begin{example}[Homogeneous sums based on {\bf G}] \label{example_Gauss}
\rm Let ${\bf G} = \{ G_i : i\geq 1 \}$ be defined as above. Without loss of generality, we can always assume that $G_i=I^G_1(h_i) = G(h_i) $ , for some isonormal Gaussian process $G=\{G(h): h\in \h \}$ based on  a real separable Hilbert space $\h$, where $\{h_i: i\geq 1 \}$ is an orthonormal system in $\h$, and $I_1^G$ denotes a Wiener-It\^o integral of order 1 with respect to $G$.   With this representation, one has that $Q_q(N,f,{\bf G})$ belongs to the so-called $q$-th {\it Gaussian Wiener chaos} of $G$. Indeed, we can write
$$Q_q(N,f,{\bf G}) = I^G_q(h), $$
where
\begin{equation}\label{e:tensorgauss}  h = \sum_{ i_1,\cdots,i_q }^N f(i_1,\cdots,i_q) h_{i_1} \otimes \cdots \otimes h_{i_q},
\end{equation}
and the symbol $\otimes$ is a usual tensor product. It is a classic result that random variables of the type $I_q^G(h)$, where $h$ is as in (\ref{e:tensorgauss}), are dense in the $q$th Wiener chaos of $G$ (see e.g. \cite[Chapter 1]{nualart}).
\end{example}

\begin{example}[Homogeneous sums based on {\bf E}] \label{example_Walsh}
\rm Fix $q\geq 1$, let $f: \mathbb{N}^q \rightarrow \R$ be a symmetric function vanishing on diagonals. We consider the Rademacher sequence ${\bf E} = \{ e_i : i\geq 1 \}$ defined above. Random variables with the form
$$ J_q =  \sum_{ i_1,\cdots,i_q } f(i_1,\cdots,i_q) e_{i_1}\cdots e_{i_q}, $$
where the series converge in $L^2(\p)$, compose the so-called \emph{$q$th Walsh chaos} of {\bf E}. (See \cite[Chapter IV]{ledoux_talagrand_book}, or Remark 2.7 in \cite{npr4}.)  In particular, let $f: [N]^q \rightarrow \R$ be a symmetric function vanishing on diagonals, then homogeneous sums of the type
$$Q_q(N,f,{\bf E}) = \sum_{ i_1,\cdots,i_q }^N f(i_1,\cdots,i_q) e_{i_1}\cdots e_{i_q}$$
are elements of $q$-th Walsh chaos of {\bf E}. Recall that the Walsh chaos enjoys the following decomposition property:
 for every $F\in L^2(\sigma({\bf E})) $ (that is, the set of
square integrable functional of the sequence ${\bf E}$), there exists a unique sequence of square-integrable symmetric functions vanishing on diagonals
$\{ f_q  : q\geq 1\}$, such that
$$ F =  E[F]+ \sum_{q\geq 1} q! \sum_{i_1<i_2<\ldots<i_q}
f_q(i_1, \ldots, i_q) e_{i_1} \cdots e_{i_q} , $$
where the double series converges in $L^2$.
\end{example}

\begin{example}[Homogeneous sums based on \textbf{P}] \label{example_Poisson}
\rm Let $\textbf{P} = \{ P_i : i\geq 1 \}$ be defined as above. Without loss of generality, we can always assume that, for every $i \geq 1$, $P_i=I_1(g_i) = I_1^{\hat{\eta}}(g_i) $, where $I_1 = I_1^{\hat{\eta}}$ indicates a single Wiener-It\^o integral with respect to a compensated Poisson measure $\hat{\eta}$ on some measurable space $(Z,\mathcal{Z}) $, with $\sigma$-finite and non-atomic control measure $\mu$. Here, $\textbf{g}=\{  g_i: i\geq 1\}$ is a collection of functions in $L^2(Z,\mathcal{Z},\mu)$ such that $g_i = {\bf 1}_{A_i}/\sqrt{\lambda_i}$, where the $\{A_i : i\geq 1\}$ are disjoint measurable sets such that $\mu(A_i) = \lambda_i$.  For instance, one may take $Z=\R_+$, $\mu=$ Lebesgue measure, $g_i=\1_{(\lambda_1+ \cdots + \lambda_{i-1} , \lambda_1+ \cdots + \lambda_{i} ]}$  for $i\geq 2$, and $g_1=\1_
{[0, \lambda_1]}$. It follows that the homogeneous sum $Q_q(N,f,\textbf{P})$ belongs to $q$-th Poisson Wiener chaos of $ \hat{\eta}$, since
$$Q_q(N,f,\textbf{P}) =I_q(g) = I^{\hat{\eta}}_q(g), $$
where
\begin{equation}\label{e:tensorpoiss}
g = \sum_{ i_1,\cdots,i_q }^N f(i_1,\cdots,i_q) g_{i_1} \otimes \cdots \otimes g_{i_q}.
\end{equation}
and $I_q = I^{\hat{\eta}}_q$ indicates a multiple Wiener-It\^o of order $q$ with respect to $\hat{\eta}$. It is well-known that random variables of the type $I_q(g)$, where $g$ is as in (\ref{e:tensorpoiss}), are dense in the $q$th Wiener chaos of $\hat{\eta}$ (see e.g. \cite[Chapter 5]{pt_book}),
\end{example}

Concerning the `universal nature' of homogeneous sum based on ${\bf G}$, the following result was proved in \cite{npr2} (see also \cite[Chapter 11]{nourdin_peccati_book}).

\begin{theorem}[Theorem 1.10 in \cite{npr2}]\label{t:unigauss} Homogeneous sums based on ${\bf G}$ are \textbf{\emph{universal}} with respect to normal approximations, in the following sense:
fix $q\geq 2$,  let $\{ N^{(n)}: n\geq 1 \}$ be a sequence  of integers going to infinity, and let $\{f^{(n)} : n\geq 1\}$ be a sequence of mappings, such that each function $f^{(n)}:[N^{(n)}]^q \rightarrow \R$ is symmetric and vanishes on diagonals. Assume that $ \E[Q_q(N^{(n)},f^{(n)},{\bf G})^2] \rightarrow 1$ as $n \rightarrow \infty$. Then, the following four properties are equivalent as $n \rightarrow \infty$.
\begin{enumerate}
  \item[\bf (1)] The sequence $\{ Q_q(N^{(n)},f^{(n)},{\bf G}) : n\geq 1\}$ converges in distribution to $Y \sim \N(0,1)$;
  \item[\bf (2)] $ \E[Q_q(N^{(n)},f^{(n)},{\bf G})^4] \rightarrow 3$;
 \item[\bf (3)]  for every sequence ${\bf X}=\{ X_i :  i\geq 1\}$ of independent  centered random variables with unit variance and such that $\sup_i\mathbb{E}|X_i|^{2+\epsilon}<\infty$, the sequence $\{ Q_q(N^{(n)},f^{(n)},{\bf X}) :  n\geq 1 \}$ converge in distribution to $Y \sim \N(0,1)$;
  \item[\bf (4)]  for every sequence ${\bf X}=\{ X_i :  i\geq 1\}$ of independent and identically distributed centered random variables with unit variance, the sequence $\{ Q_q(N^{(n)},f^{(n)},{\bf X}) :  n\geq 1 \}$ converge in distribution to $Y \sim \N(0,1)$.
\end{enumerate}
\end{theorem}

For several applications of Theorem \ref{t:unigauss} in random matrix theory, see \cite{NouPecMat}. The following negative result concerns homogeneous sums based on ${\bf E}$.
\begin{proposition}\label{p:walshnot}
Homogeneous sums inside the Walsh chaos are \textbf{\emph{not universal}} with respect to normal approximations.
\end{proposition}
\begin{proof}
To show this assertion, we present the counterexample described in \cite[p. 1956]{npr2}. Let  {\bf G} and {\bf E} be defined as above. Fix $q\geq 2$. For each $N \geq q$, we set
$$ f_N(i_1, i_2, \ldots, i_q)\!=\! \left\{
          \begin{array}{ll}
            1/(q! \sqrt{N-q+1}), &\!\! \hbox{ if } \{ i_1, i_2, \ldots, i_q \} \!=\! \{ 1,2, \ldots, q-1, s\} \text{ for  } q\leq s \leq N ;\\
           0 , & \hbox{ otherwise.}
          \end{array}
        \right.
 $$
The homogeneous sum thus defined is $$Q_q(N,f_N,{\bf E}) = e_1 e_2\cdots e_{q-1} \sum_{i=q}^{N} \cfrac{e_i}{\sqrt{N-q+1}} ,$$
with $\E[Q_q(N,f_N,{\bf E})]= 0 $ and
$\var[Q_q(N,f_N,{\bf E})]= 1 $.
Since $e_1 e_2\cdots e_{q-1} $ is a random sign independent of $\{ e_i: i\geq q \}$, we have that $Q_q(N,f_N,{\bf E}) \overset{\rm law}{\longrightarrow} \N(0,1)$, as $N\rightarrow \infty$, by virtue of the usual CLT. However, for every $N\geq 2$, one has that $Q_q(N,f_N,{\bf G}) \stackrel{\rm law}{=} G_1 G_2\cdots G_q$. Since $G_1 G_2\cdots G_q$ is not Gaussian for every $q\geq 2$, we deduce that $Q_q(N,f_N,{\bf G})$ does not converge in distribution to a normal random variable. \\
\end{proof}

\medskip

The principal aim of this paper is to provide a positive answer to the following question.
\begin{problem}
Are homogeneous sums based on ${\bf P}$ universal with respect to normal approximations? In other words: can we replace ${\bf G}$ with ${\bf P}$ inside the statement of Theorem \ref{t:unigauss}?
\end{problem}

We will see in Section \ref{sec_thm} that the answer is positive both in the one-dimensional and multi-dimensional cases. Our techniques are based on the tools developed in \cite{pstu,peczheng}, that are in turn recent developments of the so-called `Malliavin-Stein method'-- given by the combination of Stein's method and Malliavin calculus.

As a by-product of our achievements, we will also prove some new CLTs on the Poisson Wiener chaos. Indeed, in the forthcoming Theorem \ref{thm_HS_Poisson} and Theorem \ref{thm_HS_Poisson_multi}, we shall show that, in the special case of elements of the Poisson Wiener chaos that are also homogeneous sums, the sufficient conditions for normal approximations established in \cite{pstu, peczheng} turn out to be also necessary. As anticipated, this yields some new examples of `fourth moment theorems' -- such as the ones proved by Nualart and Peccati in \cite{nuapec} (see also Chapter 5 in \cite{nourdin_peccati_book}). Other `fourth moment theorems' in a Poisson setting can be found in \cite{PLR1}.

\smallskip

The paper is organized as follows.  In Section \ref{sec_prelim}, we discuss some preliminaries, including multiple Wiener-It\^o integrals on the Poisson space, product formulae and star contractions. In Section \ref{sec_thm} we present the main results, in both the one-dimensional and multi-dimensional cases, and demonstrate the universal nature of homogeneous sums inside the Poisson Wiener chaos. Section \ref{sec_proofs} is devoted to an important technical proposition as well as to the proofs of our main results.

\end{section}

\begin{section}{Some preliminaries} \label{sec_prelim}
\subsection{Poisson measures and integrals}
Let $(Z,\mathcal{Z},\mu) $ be a measure space such that
$Z$ is a Borel space and $\mu$ is a $\sigma$-finite non-atomic
Borel measure. We set $\mathcal{Z}_{\mu} = \{ B\in \mathcal{Z}: \mu(B)< \infty \}$. In what follows, we write
$\hat{\eta} = \{\hat{\eta}(B) : B\in \mathcal{Z}_{\mu} \} $ to indicate a {\it compensated Poisson measure} on $(Z,\mathcal{Z}) $ with {\sl control} $\mu$. In other words, $\hat{\eta} $ is a collection of random variables defined on some probability space $(\Omega, \mathcal{F}, \p) $, indexed by
the elements of $\mathcal{Z}_{\mu} $ and such that: (i) for every $B,C \in \mathcal{Z}_{\mu}$ such that $B \cap C = \varnothing$, the random variables $ \hat{\eta}(B)$ and $ \hat{\eta}(C)$ are independent;  (ii) for every $B \in \mathcal{Z}_{\mu} $, $\hat{\eta}(B) \stackrel{\rm law}{=} \eta(B)-\mu(B) $, where $\eta(B) $ is a Poisson random variable with paremeter $\mu(B) $. A random measure verifying property (i) is customarily called `completely random' or, equivalently, `independently scattered' (see e.g. the monograph \cite{pt_book} for a detailed discussion of these concepts).\\

In order to simplify the forthcoming discussion, we shall make use of the following conventions:
\begin{itemize}
  \item[--] We shall write interchangeably $\sum\limits_{1\leq i_1,\cdots,i_q\leq N}$ and $\sum\limits_{i_1,\cdots,i_q}^N$.
  \item[--] For every $k\geq 1$ and $f\in L^k(Z^q,\mathcal{Z}^q,\mu^q) := L^k(\mu^q)$, we write $\|f\|_{L^k}$ to indicate the norm $\|f\|_{L^k(\mu^q)}$.
  \item[--] For every $q\geq 2$, the class $L^k_s(\mu^q)$ as defined is the subspace of $L^k(\mu^q)$ of functions that are $\mu^q$-almost everywhere symmetric; also, one customarily writes $L^k_s(\mu^1) = L^k(\mu^1)=L^k_s(\mu) = L^k(\mu)$.
  \item[--] For any positive integer $N$, $[N]$ stands for the set $\{ 1,2,\cdots,N \}$.
  \item[--] For any two functions $f,g \in L^2(\mu)$, $f \otimes g$ is the
      tensor product of $f$ and $g$, that is, $f \otimes g(x,y)= f(x)g(y)$. Iterated tensor products of the type $f_1\otimes f_2\otimes \cdots \otimes f_q(x_1,\ldots,x_q)$ are defined by recursion.
\end{itemize}

\begin{definition}
For every deterministic function $h\in L^2(\mu)$, we write
\[I_1(h)=\hat{\eta}(h) = \int_Z h(z) \hat{\eta}(dz)\] to indicate the {\bf Wiener-It\^o
integral} of $h$ with respect to $\hat{\eta}$. For every $q\geq 2$ and every $f\in L^2_s(\mu^q)$, we denote by $I_q(f)$
the {\bf multiple Wiener-It\^o integral}, of order $q$, of $f$ with respect to $\hat{\eta}$. We also set $I_q(f)=I_q(\tilde{f})$, for every $f\in L^2(\mu^q)$, and $I_0(C)=C$ for every constant $C$. Here, $\tilde{f}$ is the symmetrization of the function $f$. For every $q\geq 1$, the collection of all random variables of the type $I_q(f)$, $f\in L^2_s(\mu^q)$, is denoted by $C_q$ and is called the $q$th {\bf Wiener chaos} of $\hat{\eta}$.
\end{definition}

We recall the following chaotic decomposition of $L^2(\sigma(\hat{\eta}))$ (that is, the space of all square-integrable functionals of $\hat{\eta}$):
\[
L^2(\sigma(\hat{\eta})) = \mathbb{R} \oplus \bigoplus_{q=1}^\infty C_q,
\]
where the symbol $\oplus$ denotes a direct sum in $L^2(\mathbb{P})$. The reader is referred e.g. to Peccati and Taqqu \cite{pt_book}, Privault \cite{priv} or Nualart and Vives \cite{nuaviv} for a complete discussion of multiple Wiener-It\^o integrals and their properties. The following proposition contains two fundamental properties that we will use in sequel.

\begin{proposition}\label{P : MWIone}
The following equalities hold for every $q,m\geq 1$, every $f\in L_s^2(\mu^q)$  and every $g\in L_s^2(\mu^m)$:
\begin{enumerate}
  \item[\rm 1.] $\E[I_q(f)]=0$,
  \item[\rm 2.] $\E[I_q(f) I_m(g)]= q!\langle f,g  \rangle_{L^2(\mu^q)} \1_{(q=m)} $
  {\rm (isometric property).}
\end{enumerate}
\end{proposition}

\begin{remark}
\rm
For every $q\geq 1$ and every $f\in L^2_s(\mu^q)$, we shall also denote by $I^G_q(f)$ the {\bf multiple Wiener-It\^o integral}, of order $q$, of $f$ with respect to an isonormal process $G$ over the Hilbert space $\mathfrak{H} = L^2(Z,\mathcal{Z},\mu)$. A detailed introduction to these objects can be found in  \cite{nourdin_peccati_book,nualart, pt_book}.
\end{remark}

\subsection{Product formulae}

In order to give a simple description of the Product formulae for multiple Poisson integrals (see formula (\ref{product})), we (formally) define a contraction kernel $f \star_r^l g$ on $Z^{p+q-r-l}$ for functions $f\in L^2_s(\mu^p) $ and $g \in L^2_s(\mu^q) $, where $p,q \geq 1$, $r=1,\ldots, p\wedge q$ and $l=1,\ldots,r $, as follows:
\begin{eqnarray}
& & f \star_r^l
g(\gamma_1,\ldots,\gamma_{r-l},t_1,,\ldots,t_{p-r},s_1,,\ldots,s_{q-r}) \label{contraction} \\
&=& \int_{Z^l} \mu^l(dz_1,\ldots,dz_l)
f(z_1,,\ldots,z_l,\gamma_1,\ldots,\gamma_{r-l},t_1,\ldots,t_{p-r}) \nonumber \\
& & \quad\quad\quad\quad\quad\quad\quad\quad\quad\quad\quad\quad \times g(z_1,,\ldots,z_l,\gamma_1,\ldots,\gamma_{r-l},s_1,\ldots,s_{q-r}). \nonumber
\end{eqnarray}
In other words, the star operator `$\,\star_r^l\,$' reduces the number of variables in the tensor product of $f$ and $g$ from $p+q$ to $p+q-r-l$: this operation is realized by first identifying $r$ variables in $f$ and $g$, and then by integrating out $l$ among them. We also use the notation
\begin{eqnarray}
& & f \otimes_r f = f \star_r^r
g(t_1,,\ldots,t_{p-r},s_1,,\ldots,s_{q-r}) \label{contraction2} \\
&=& \int_{Z^r} \mu^l(dz_1,\ldots,dz_r)
f(t_1,\ldots,t_{p-r},z_1,\ldots,z_r) \times g(s_1,\ldots,s_{q-r},z_1,\ldots,z_r). \nonumber
\end{eqnarray}
The operator $f \otimes_r f$ and its symmetrization $f \tilde{\otimes}_r f$ play a fundamental role in the derivation of limit theorems inside the Gaussian Wiener-It\^o chaos, see e.g. \cite{nourdin_peccati_book, nuapec}.\\

We present here an important {\it product formula} for Poisson multiple integrals (see e.g. \cite{kab, pt_book, surg2} for a proof).
\begin{proposition}
[Product formula] Let $f\in L^2_s(\mu^p) $ and $g\in
L^2_s(\mu^q)$, $p,q\geq 1 $, and suppose moreover that $f \star_r^l g
\in L^2(\mu^{p+q-r-l})$ for every $r=1,\ldots,p\wedge q $ and $
l=1,\dots,r$ such that $l\neq r $. Then,
\begin{equation} \label{product}
I_p(f)I_q(g) = \sum_{r=0}^{p\wedge q} r!
\left(
\begin{array}{c}
  p\\
  r\\
\end{array}
\right)
 \left(
\begin{array}{c}
  q\\
  r\\
\end{array}
\right)
 \sum_{l=0}^r
 \left(
\begin{array}{c}
  r\\
  l\\
\end{array}
\right)  I_{p+q-r-l} \left(\widetilde{f\star_r^l g}\right),
\end{equation}
 with the tilde $\sim$ indicating a symmetrization, that
is,
$$\widetilde{f\star_r^l g}(x_1,\ldots,x_{p+q-r-l})=\cfrac{1}{(p+q-r-l)!}
\sum_\sigma f\star_r^l g(x_{\sigma(1)},\ldots,x_{\sigma(p+q-r-l)}), $$
where $\sigma $ runs over all $(p+q-r-l)! $ permutations of the
set $\{1,\ldots,p+q-r-l \}$.

\end{proposition}

Fix integers $p,q \geq 0$ and $|q-p| \leq k \leq p+q$, consider two kernels $f\in L^2_s(\mu^p)$ and $g \in L^2_s(\mu^q)$, and recall the multiplication formula (\ref{product}). We will now introduce an operator
$G_k^{p,q}$, transforming the function $f$, of $p$ variables,
and the function $g$, of $q$ variables, into a function $G_k^{p,q}(f,g)$, of
$k$ variables. More precisely, for $p,q,k $ as above, we define the function $(z_1,\ldots,z_k)\mapsto G_k^{p,q}(f,g)(z_1,\ldots,z_k)$, from $Z^k$ into $\R$, as follows:
\begin{equation}\label{force}
G_k^{p,q}(f,g)(z_1,\ldots,z_k) \!=\! \sum_{r=0}^{p\wedge q} \sum_{l=0}^{r}
\1_{(p+q-r-l=k)} r!
\left(
\begin{array}{c}
  p\\
  r\\
\end{array}
\right)
 \left(
\begin{array}{c}
  q\\
  r\\
\end{array}
\right)
 \left(
\begin{array}{c}
  r\\
  l\\
\end{array}
\right)  \widetilde{f\star_r^l g}(z_1,\ldots,z_k) ,
\end{equation}
where the tilde $\sim$ means symmetrization, and the star contractions are defined in formula (\ref{contraction}) and the subsequent discussion. Observe the following three special cases: (i) when $p=q=k=0$, then $f$ and $g$ are both real constants, and $G_0^{0,0}(f,g) = f\times g$, (ii) when $ p=q\geq 1$ and $k=0$, then $G_0^{p,p}(f,g) = p!\langle f,g \rangle_{L^2(\mu^p)} $, (iii) when $p=k=0$ and $q>0$ (then, $f$ is a constant), $G_0^{0,p}(f,g)(z_1,\ldots,z_q) = f\times g(z_1,\ldots,z_q)$. By using this notation, (\ref{product}) becomes
\begin{equation} \label{product2}
I_p(f)I_q(g) = \sum_{k=|q-p|}^{p+q} I_k(G_k^{p,q}(f,g)).
\end{equation}
The advantage of representation (\ref{product2}) (as opposed to (\ref{product})) is
that the RHS of (\ref{product2}) is an \emph{orthogonal sum}, a feature that will simplify the computations to follow. \\

\end{section}

\begin{section}{Main results} \label{sec_thm}

\begin{subsection}{One-dimensional case: fourth moments and universality}
We recall the following theorem, first proved in \cite{nuapec}, stating that the convergence in law of a sequence of Gaussian Wiener integrals towards a normal distribution can be characterized by their variances and fourth moments. See \cite[Chapter 5]{nourdin_peccati_book} for a detailed discussion, as well as examples and bibliographic remarks.

\begin{theorem}[See \cite{no, nuapec}]
\label{thm_equiv_Gauss_copy}
Fix $q\geq 2$, let $h^{(n)}\in L^2_s(\mu^q)$, $n\geq 1$, and let
$$Z^{(n)} = I^G_q(h^{(n)}), \quad n\geq 1,$$
 be a sequence of random variables having the form of a multiple Wiener-It\^o integral of order $q$, of $h^{(n)}$ with respect to an isonormal Gaussian process $G$ over the Hilbert space $\mathfrak{H} = L^2(\mu)$. Assume that $\lim_{n\to\infty} \var(Z^{(n)})=\lim_{n\to\infty} \E\left[\big(Z^{(n)}\big)^2 \right]=1$. Then, the following three assertions are equivalent as $n\rightarrow \infty$:
\begin{itemize}
  \item[\bf (1)] $Z^{(n)} \overset{\rm law}{\longrightarrow} Y\sim \mathcal{N}(0,1)$;
  \item[\bf (2)] $\E\left[\big(Z^{(n)}\big)^4 \right] \rightarrow \E[Y^4] = 3$ ;
  \item[\bf (3)] $\forall r=1,\ldots, q-1$, $\| h^{(n)} \otimes_r  h^{(n)} \|_{L^2} \rightarrow 0$, where the contraction $\otimes_r = \star_r^r$ is defined according to {\rm (\ref{contraction2})}.
\end{itemize}
\end{theorem}

Extending Theorem \ref{thm_equiv_Gauss_copy} to multiple integrals with respect to a Poisson measure is a demanding task, since the product formula (\ref{product}) (which is more involved than in the Gaussian case) quickly leads to some inextricable expressions for moments of order four. A partial `fourth moment theorem' can be found in \cite[Theorem 2]{pt}, in the special case of double Poisson integrals. We now present an exact analogous of Theorem \ref{thm_equiv_Gauss_copy} for homogeneous sums inside a fixed Poisson Wiener chaos. Its proof, together with the one of the subsequent Theorem \ref{thm_universality_Poisson_1}, is deferred to Section \ref{sec_proofs}.

\begin{theorem}[Fourth moment theorem for Poisson sums] \label{thm_HS_Poisson}
Let $\{ \lambda_i : i\geq 1\}$ be a collection of positive real numbers, and assume that $\inf\limits_{i\geq 1} \lambda_i = \alpha >0$. Let ${\bf P}=\{ P_i : i\geq 1\}$ be a collection of independent random variables verifying {\rm (\ref{e:pi})}. Fix an integer $q\geq 1$. Let $\{ N^{(n)}, f^{(n)}: n\geq 1\}$ be a double sequence such that  $\{ N^{(n)} : n\geq 1\}$ is a sequence of integers diverging to infinity, and each $f^{(n)}: [N^{(n)}]^{q} \rightarrow \R$ is symmetric and vanishes on diagonals. We set
$$ F^{(n)} = Q_q(N^{(n)},f^{(n)}, {\bf P}) = \sum_{i_1, \cdots, i_q}^{N^{(n)}} f^{(n)}(i_1, \cdots, i_q) P_{i_1} \cdots P_{i_q} = I_q(g^{(n)}) , $$
where
$$g^{(n)} = \sum_{i_1,\cdots,i_q}^{N^{(n)}} f^{(n)}(i_1,\cdots,i_q) g_{i_1} \otimes \cdots \otimes g_{i_q} ,\quad n\geq 1,$$
and the representation of $F^{(n)}$ as a multiple Wiener-It\^o integral is the same as in Example {\rm \ref{example_Poisson}}. Suppose that  $ \E\left[\big(F^{(n)}\big)^2 \right] \rightarrow \sigma^2\in (0,\infty) $.
 Then, the following two statements are equivalent, as $n\to\infty$:
\begin{itemize}
  \item[\bf (1)] $F^{(n)} \overset{\rm law}{\longrightarrow}Y\sim \mathcal{N}(0,\sigma^2)$;
  \item[\bf (2)] $\E\left[\big(F^{(n)}\big)^4 \right] \rightarrow \E[Y^4] = 3 \sigma^4$.
 \end{itemize}
When $q=1$, either one  of conditions {\bf (1)}--{\bf (2)} is equivalent to
\begin{itemize}
\item[\bf (3a)]  $\sum_{i=1}^{N^{(n)}} f^{(n)}(i)^4\frac{1}{\lambda_i} \rightarrow 0.$
\end{itemize}
Finally, when $q\geq 2$, either one of conditions {\bf (1)}--{\bf (2)} is equivalent to either one of the following two equivalent conditions {\bf (3b)}--{\bf (3b')}
\begin{itemize}
\item[\bf (3b)]  $\int_{Z^q} \big(g^{(n)} \big)^4 \rightarrow 0$ and
$\forall r=1,\cdots,q$, $\forall l=1,\cdots,r\wedge (q-1)$,
      $ \| g^{(n)} \star_r^l g^{(n)} \|_{L^2} \rightarrow 0 $, where the star contractions $\star_r^l$ are defined according to {\rm (\ref{contraction})};
      \item[\bf (3b')]
$\forall r=1,\cdots,q-1$,
      $ \| g^{(n)} \star_r^r g^{(n)} \|_{L^2}=\| g^{(n)} \otimes_r g^{(n)} \|_{L^2} \rightarrow 0 $.

\end{itemize}
\end{theorem}

\begin{remark}{\rm The assumption $\inf\limits_{i\geq 1} \lambda_i >0$ is necessary for proving the two implications: {\bf (1)} $\Rightarrow$ {\bf (2)} and {\bf (3b')} $\Rightarrow$ {\bf (3b).}
}
\end{remark}

The next statement, that will be proved by means of Theorem \ref{thm_HS_Poisson}, establishes the universal nature of Poisson homogeneous sums of order $q\geq 2$.

\begin{theorem}[Universality of the Poisson Wiener chaos]\label{thm_universality_Poisson_1} Let the sequence  ${\bf P}$ verify the same assumptions as in Theorem {\rm \ref{thm_HS_Poisson}}. Fix $q\geq 2$,  let $\{ N^{(n)}: n\geq 1 \}$ be a sequence  of integers going to infinity, and let $\{f^{(n)} : n\geq 1\}$ be a sequence of mappings, such that each function $f^{(n)}:[N^{(n)}]^q \rightarrow \R$ is symmetric and vanishes on diagonals. Assume that $ \E[Q_q(N^{(n)},f^{(n)},{\bf P})^2] \rightarrow 1$ as $n \rightarrow \infty$. Then, the following four properties are equivalent, as $n \rightarrow \infty$.
\begin{enumerate}
  \item[\bf (1)] The sequence $\{ Q_q(N^{(n)},f^{(n)},{\bf P}) : n\geq 1\}$ converges in distribution to $Y \sim \N(0,1)$;
  \item[\bf (2)] $ \E[Q_q(N^{(n)},f^{(n)},{\bf P})^4] \rightarrow 3$;
 \item[\bf (3)]  for every sequence ${\bf X}=\{ X_i :  i\geq 1\}$ of independent  centered random variables with unit variance and such that $\sup_i\mathbb{E}|X_i|^{2+\epsilon}<\infty$, the sequence $\{ Q_q(N^{(n)},f^{(n)},{\bf X}) :  n\geq 1 \}$ converge in distribution to $Y \sim \N(0,1)$;
  \item[\bf (4)]  for every sequence ${\bf X}=\{ X_i :  i\geq 1\}$ of independent and identically distributed centered random variables with unit variance, the sequence $\{ Q_q(N^{(n)},f^{(n)},{\bf X}) :  n\geq 1 \}$ converge in distribution to $Y \sim \N(0,1)$.
\end{enumerate}
\end{theorem}

\begin{remark}\label{r:cex}{\rm Theorem \ref{thm_universality_Poisson_1} is false in general for $q=1$, as one can see by considering the case $N^{(n)}=n$, $\lambda_i = i$, and $f_n$ such that $f_n(n) = 1$ and $f_n(i)=0$ for $i\neq n$. On the other hand, one can prove an equivalent of Theorem \ref{thm_universality_Poisson_1} for $q=1$, by assuming in addition that $\sup_i \lambda_i <\infty$ and by applying the standard Lindberg's CLT (see e.g. \cite[Theorem 9.6.1]{Dudley book}). The details are left to the reader.}
\end{remark}

We conclude this section with a result implying that the Wasserstein distance metrizes the convergence to Gaussian for any sequence of homogeneous sums based on a Poisson field. Recall that, given random variables $X,Y\in L^1(\mathbb{P})$, the {\it Wasserstein distance} between the law of $X$ and the law of $Y$ is defined as the quantity
\[
d_W(X,Y) = \sup_{f\in {\rm Lip}(1)} \big|\E[f(X)] -\E[f(Y)]\big|,
\]
where ${\rm Lip}(1)$ indicates the class of Lipschitz real-valued function with Lipschitz constant $\leq 1$. It is well-known that the topology induced by $d_W$, on the class of probability measures on the real line, is strictly stronger than the one induced by the convergence in distribution.

\begin{proposition}
Let the sequence of homogeneous sums $\{ F^{(n)} : n\geq 1\}$ satisfy the assumptions of Theorem {\rm \ref{thm_HS_Poisson}}. If $F^{(n)}$ converges in distribution to $Y \sim \mathcal{N}(0,1)$, as $n\to\infty$, then necessarily $d_W(F^{(n)}, Y)\to 0$.
\end{proposition}

\begin{proof} Using Corollary 3.4 (for the case $q=1$) and Theorem 4.1 (for the case $q\geq 2$) in \cite{pstu}, we see that, if conditions {\bf (3a)}-{\bf (3b)} are verified, then $d_W(F^{(n)}, Y)\to 0$, so that the conclusion follows from Theorem \ref{thm_HS_Poisson}.
\end{proof}

\end{subsection}
\begin{subsection}{Multi-dimensional case}

We now present some multidimensional extensions of the results presented in the previous section: the proofs are similar to those of the results in the previous section, and are mostly left to the reader. Our starting point is the following multi-dimensional extension of Theorem \ref{thm_equiv_Gauss_copy}, first proved by Peccati and Tudor in \cite{ptudor}. For details and generalizations, see \cite{npr, no}.
\begin{theorem}\label{thm_equiv_Gauss_multi_copy}
Let $G$ be an isonormal Gaussian process over the Hilbert space $\mathfrak{H} = L^2(\mu)$. Fix $d\geq2$ and let $C = \{C(i, j) : i, j = 1, \ldots, d\}$ be a $d\times d$
positive definite matrix. Fix integers $1 \leq q_1 \leq \cdots \leq q_d$. For any $n \geq 1$ and $i = 1, \ldots , d$, let
$h^{(n)}_i$ belong to $L_s^2(\mu^{q_i})$. Assume that
$$F^{(n)} = (F^{(n)}_1 , \ldots , F^{(n)}_d ) := (I^G_{q_1}(h^{(n)}_1 ), \ldots , I^G_{q_d}(h^{(n)}_d )) \qquad n \geq 1, $$
is such that
$$ \lim_{n\rightarrow \infty} \E[F^{(n)}_i F^{(n)}_j ] = C(i, j), \quad 1 \leq i, j \leq d. $$
Then, as $n\rightarrow \infty$, the following four assertions are equivalent:
\begin{itemize}
  \item[\bf (1)] The vector $F^{(n)}$ converges in distribution to a $d$-dimensional Gaussian vector $\Nd(0,C)$;
  \item[\bf (2)] for every $1 \leq i \leq d$, $\E\left[ (F^{(n)}_i )^4 \right] \rightarrow 3C(i, i)^2 $;
  \item[\bf (3)] for every $1 \leq i \leq d$ and every $1 \leq r \leq q_i-1$ , $\|h^{(n)}_i \otimes_r h^{(n)}_i \|_{L^2} \rightarrow 0$;
  \item[\bf (4)] for every $1 \leq i \leq d$, $F^{(n)}_i$ converges in distribution to a centered Gaussian random
variable with variance $C(i, i)$.
\end{itemize}
\end{theorem}
Combining the previous Theorem \ref{thm_HS_Poisson} with \cite[Theorem 5.8]{peczheng} we deduce the following analogue of Theorem \ref{thm_equiv_Gauss_multi_copy} for homogeneous sums inside the Poisson Wiener chaos.

\begin{theorem} \label{thm_HS_Poisson_multi}
Let $\{ \lambda_i:  i\geq 1\}$ be a collection of positive real numbers, and assume that $\inf\limits_i \lambda_i = \alpha >0$. Let ${\bf P}=\{ P_i : i\geq 1\}$ be a collection of independent random variables such that $\forall i$, $P_i$ verifies relation {\rm (\ref{e:pi})}. Fix integers $d\geq 1$ and $q_d\geq \cdots\geq  q_1 \geq 1$. Let $\{ N_j^{(n)}, f_j^{(n)}: j=1,\cdots,d; \, n\geq 1\}$ be such that for every fixed $j$, $\{ N_j^{(n)} : n\geq 1\}$ is a sequence of integers going to infinity, and each $f_j^{(n)}: [N_j^{(n)}]^{q_j} \rightarrow \R$ is symmetric and vanishes on diagonals. We consider a sequence of random vectors $F^{(n)} = (F_1^{(n)}, \cdots,  F_d^{(n)})$, $n\geq 1$, where for every $1\leq j \leq d$,
$$ F_j^{(n)} = Q_{q_j}(N_j^{(n)}, f_j^{(n)} , {\bf P}) = \sum_{i_1,\cdots,i_{q_j}}^{N_j^{(n)}} f_j^{(n)}(i_1,\cdots,i_{q_j}) P_{i_1} \cdots P_{i_{q_j}} = I_{q_j}(g_j^{(n)}) $$
with
$$g_j^{(n)} = \sum_{i_1,\cdots,i_{q_j}}^{N_j^{(n)}} f_j^{(n)}(i_1,\cdots,i_{q_j}) g_{i_1} \otimes \cdots \otimes g_{i_{q_j}} ,$$
and the representation of $ F_j^{(n)}$ as a multiple integral is the same as in Example {\rm \ref{example_Poisson}}. Given a $d\times d$ positive definite matrix $C=\Big( C(i,j) \Big)_{i,j = 1,\ldots,d}$, suppose that  $ \lim_{n\to\infty} \E[F_i^{(n)} F_j^{(n)}] \rightarrow C(i,j) $, for every $i,j=1,\ldots,d$. Then,
the following four statements are equivalent, as $n\to \infty$:
\begin{itemize}
  \item[\bf (1)] $F^{(n)} \overset{\rm law}{\longrightarrow} (Y_1,\ldots,Y_d)\sim \mathcal{N}_d(0,C)$, where $\mathcal{N}_d(0,C)$ indicates a $d$-dimensional Gaussian distribution with covariance matrix $C$;
  \item[\bf (2)] for each $j=1,\cdots,d$, $\E[(F_j^{(n)})^4] \rightarrow 3 C(j,j)^2$;
  \item[\bf (3)] for each $j=1,\cdots,d$ such that $q_j\geq 2$,
 $\forall r=1,\cdots,q_j-1$,
  \[ \| g_j^{(n)} \star_r^r g_j^{(n)} \|_{L^2} = \| g_j^{(n)} \otimes_r g_j^{(n)} \|_{L^2}  \rightarrow 0 ,\] and, for every $j$ such that $q_j=1$, \[\sum_{i=1}^{N^{(n)}} f_j^{(n)}(i)^4\frac{1}{\lambda_i} \rightarrow 0;\]
  \item[\bf (4)] for each $j=1,\cdots,d$, $F^{(n)}_j \overset{\rm law}{\longrightarrow} \mathcal{N}(0,C(j,j))$.
  \end{itemize}
\end{theorem}

Finally, we present a multi-dimensional analogous of the universality statement contained in Theorem \ref{thm_universality_Poisson_1}: it is deduced by combining Theorems \ref{thm_equiv_Gauss_multi_copy} and \ref{thm_HS_Poisson_multi} above with \cite[Theorem 7.1]{npr2}

\begin{theorem}[Multi-dimensional Universality] \label{thm_universality_Poisson_d}
Let the assumptions and notations of Theorem {\rm \ref{thm_HS_Poisson_multi}} prevail. Then, the following two assertions are equivalent as $n\to\infty$:

\begin{enumerate}
  \item[\bf (1)] The sequence $\{ F^{(n)} : n\geq 1\}$ converges in distribution to $(Y_1,\ldots,Y_d) \sim \N_d(0,1)$;
 \item[\bf (2)]  for every sequence ${\bf X}=\{ X_i :  i\geq 1\}$ of independent  centered random variables with unit variance and such that $\sup_i\mathbb{E}|X_i|^{3}<\infty$, the sequence of $d$-dimensional vectors
 \[
 \{ Q_q(N_j^{(n)},f_j^{(n)},{\bf X}): j=1,\ldots,d \}, \quad n\geq 1,
 \]
 converges in distribution to $(Y_1,\ldots Y_d)$.
\end{enumerate}

\end{theorem}

\end{subsection}
\end{section}

\begin{section}{Proofs} \label{sec_proofs}
\begin{subsection}{A technical result}

The following technical statement is the key to our main results.

\begin{proposition} \label{prop_ineq_star}
Let the notation and assumptions of Theorem {\rm \ref{thm_HS_Poisson}} prevail, and fix $q\geq 2$. If $\forall p=1, 2, \ldots, q-1$, one has $\lim_{n\to\infty} \| g^{(n)} \star_p^p g^{(n)} \|_{L^2} = 0$, then, as $n\rightarrow \infty$:
\begin{itemize}
  \item[\bf (a)] $\int_{Z^q} \big(g^{(n)} \big)^4 \rightarrow 0$ ;
  \item[\bf (b)] $\forall r=1,\cdots,q$, $\forall l=1,\cdots,r\wedge (q-1)$,
      $ \| g^{(n)} \star_r^l g^{(n)} \|_{L^2} \rightarrow 0 $.
\end{itemize}
\end{proposition}
\begin{proof}
In the following proof, we shall write $\sum_{i_1,\cdots,i_p}$ to indicate $\sum_{i_1,\cdots,i_p}^{N^{(n)}}$. \\
For $p=1, 2, \ldots, q-1$,
\begin{eqnarray*}
  g^{(n)} \star_p^p g^{(n)}
 &=& \sum_{i_1,\cdots,i_q}
 \sum_{j_1,\cdots,j_q} f^{(n)}(i_1,\cdots,i_q) f^{(n)}(j_1,\cdots,j_q) \\
 & & \qquad \times( g_{i_1} \otimes \cdots \otimes g_{i_q} )
 \star_p^p ( g_{j_1} \otimes \cdots \otimes g_{j_q} ) \\
 &=& \sum_{a_1,\cdots,a_p}  \Big(
     \sum_{i_1,\cdots,i_{q-p}} \sum_{j_1,\cdots,j_{q-p}}  \prod_{l=1}^{p} \| g_{a_l} \|^2_{L^2}\times f^{(n)}(a_1,\cdots,a_p, i_1,\cdots,i_{q-p}) \\
 & & \qquad \times f^{(n)}(a_1,\cdots,a_p, j_1,\cdots,j_{q-p})
     \, g_{i_1} \otimes \cdots \otimes g_{i_{q-p}} \otimes
     g_{j_1} \otimes \cdots \otimes g_{j_{q-p}}
\Big) \\
 &=& \sum_{k_1,\cdots,k_{2q-2p}} \sum_{a_1,\cdots,a_p}  f^{(n)}(a_1,\cdots,a_p, k_1,\cdots,k_{q-p})
f^{(n)}(a_1,\cdots,a_p, k_{q-p+1},\cdots,k_{2q-2p}) \\
 & & \qquad \times \prod_{l=1}^{p} \| g_{a_l} \|^2_{L^2}\times  g_{k_1} \otimes \cdots \otimes g_{k_{2q-2p}} \\
 &=& \sum_{k_1,\cdots,k_{2q-2p}} \sum_{a_1,\cdots,a_p}  f^{(n)}(a_1,\cdots,a_p, k_1,\cdots,k_{q-p})
f^{(n)}(a_1,\cdots,a_p, k_{q-p+1},\cdots,k_{2q-2p}) \\
 & & \qquad \quad \times  g_{k_1} \otimes \cdots \otimes g_{k_{2q-2p}} \quad ,
\end{eqnarray*}
from which we deduce that
\begin{eqnarray}
 \| g^{(n)} \star_p^p g^{(n)} \|^2_{L^2}
\!&=&\! \sum_{k_1,\cdots,k_{2q-2p}} \Big( \sum_{ a_1,\cdots,a_p}
 f^{(n)}(a_1,\cdots,a_p, k_1,\cdots,k_{q-p}) \nonumber \\
 & &\quad\quad \qquad \times f^{(n)}(a_1,\cdots,a_p, k_{q-p+1},\cdots,k_{2q-2p})
\Big)^2.  \label{norm_L2_h_n_star_p}
\end{eqnarray}

\noindent We first prove {\bf (a)}. Using the definition of the functions $\{g_i : i\geq 1\}$,
\begin{eqnarray*}
\big(g^{(n)} \big)^4 &=& \sum_{i_1,\cdots,i_q} \sum_{j_1,\cdots,j_q}
 \sum_{k_1,\cdots,k_q} \sum_{s_1,\cdots,s_q}
 f^{(n)}(i_1,\cdots,i_q) f^{(n)}(j_1,\cdots,j_q) f^{(n)}(k_1,\cdots,k_q) f^{(n)}(s_1,\cdots,s_q) \\
 & & \qquad \times (g_{i_1} \otimes \cdots \otimes g_{i_q} ) \times ( g_{j_1} \otimes \cdots \otimes g_{j_q} )
   \times ( g_{k_1} \otimes \cdots \otimes g_{k_q} ) \times ( g_{s_1} \otimes \cdots \otimes g_{s_q} ) \\
   &=& \sum_{i_1,\cdots,i_q}  \big(f^{(n)} \big)^4(i_1,\cdots,i_q)\,  g_{i_1} \otimes \cdots \otimes g_{i_q} \times \prod_{l=1}^q \frac{1}{\lambda^{3/2}_{i_l}},
\end{eqnarray*}
yielding
\begin{eqnarray*}
&& \int \big(g^{(n)} \big)^4 d\mu^q = \sum_{i_1,\cdots,i_q}  \big( f^{(n)} \big)^4(i_1,\cdots,i_q) \prod_{l=1}^{q}  \frac{1}{\lambda_{i_l}}
  \leq  \frac{1}{\alpha^q}  \sum_{i_1,\cdots,i_q}  \big(f^{(n)} \big)^4(i_1,\cdots,i_q).
\end{eqnarray*}
Now, specializing formula (\ref{norm_L2_h_n_star_p}) to the case $p=q-1$, we deduce
\begin{eqnarray*}
\| g^{(n)} \star_{q-1}^{q-1} g^{(n)} \|_{L^2}^2 &=& \sum_{k_1, k_2} \Big( \sum_{ a_1,\cdots,a_{q-1}}
 f^{(n)}(a_1,\cdots,a_{q-1},k_1) \\
 & & \qquad \times f^{(n)}(a_1,\cdots,a_{q-1},k_2) \Big)^2  \\
&\geq& \sum_{k} \Big( \sum_{a_1,\cdots,a_{q-1}}
 \big(f^{(n)} \big)^2(a_1,\cdots,a_{q-1},k) \Big)^2 \\
&=& \!\!\!\sum_{ a_1,\cdots,a_{q-1}} \sum_{ b_1,\cdots,b_{q-1}}\Big(  \sum_{k} \big(f^{(n)} \big)^2(a_1,\cdots,a_{q-1},k) \big(f^{(n)} \big)^2(b_1,\cdots,b_{q-1},k) \Big) \\
&\geq& \sum_{ a_1,\cdots,a_{q-1}} \big(f^{(n)} \big)^4(a_1,\cdots,a_q) \\
&\geq& \int \Big( g^{(n)} \Big)^4 d \mu^q \times \alpha^q ,
\end{eqnarray*}
which proves {\bf (a)}, since $\alpha= \inf\limits_i \lambda_i  >0$ by assumption.\\

\noindent The proof of {\bf (b)} consists of two steps.\\
\noindent {\bf (b1)} Let $r=q$. For any $l\in \{ 1,\cdots,q-1 \}$, we have,
\begin{eqnarray*}
& & g^{(n)} \star_q^l g^{(n)} \\
&=& \sum_{i_1,\cdots,i_q} \sum_{j_1,\cdots,j_q} f^{(n)}(i_1,\cdots,i_q)
f^{(n)}(j_1,\cdots,j_q) [g_{i_1} \otimes \cdots \otimes g_{i_q}] \star_r^l [g_{j_1} \otimes \cdots \otimes g_{j_q}] \\
&=& \sum_{a_1,\cdots,a_l}     \sum_{b_1,\cdots,b_{q-l}} g_{b_1} \otimes \cdots \otimes g_{b_{q-l}}
\times f^2(a_1,\cdots,a_l,b_1,\cdots,b_{q-l}) \prod_{t=1}^{q-l} \lambda_{b_t}^{-1/2} \\
&=& \sum_{b_1,\cdots,b_{q-l}} g_{b_1} \otimes \cdots \otimes g_{b_{q-l}}\prod_{t=1}^{q-l} \lambda_{b_t}^{-1/2}\, \Big( \sum_{a_1,\cdots,a_l}  f^2(a_1,\cdots,a_l,b_1,\cdots,b_{q-l})
\Big).
\end{eqnarray*}
These equalities lead to the estimate
\begin{eqnarray*}
 \| g^{(n)} \star_q^l g^{(n)} \|^2_{L^2}
&=& \sum_{b_1,\cdots,b_{q-l}} \prod_{t=1}^{q-l} \lambda^{-1}_{b_t}
\Big( \sum_{a_1,\cdots,a_l} f^2(a_1,\cdots,a_l,b_1,\cdots,b_{q-l})
\Big)^2 \\
&\leq& \frac{1}{\alpha^{q-l}} \| g^{(n)} \star_l^l g^{(n)} \|_{L^2}^2 ,
\end{eqnarray*}
yielding (since $\alpha>0$) that $\| g^{(n)} \star_l^l g^{(n)} \|_{L^2}\to0$ implies $\| g^{(n)} \star_q^l g^{(n)} \|_{L^2}\to 0$. \\

\noindent {\bf (b2)} For any $r=1,\cdots, q-1$, and $l=1,\cdots,r$, we see that
\begin{eqnarray*}
& &  g^{(n)} \star_r^l g^{(n)}  \\
&=& \sum_{a_1,\cdots,a_l}
\left[ \sum_{b_1,\cdots,b_{r-l}} \prod_{u=1}^{r-l} \lambda^{-1/2}_{b_u} g_{b_1} \otimes \cdots \otimes g_{b_{r-l}} \right]
\sum_{i_1,\cdots,i_{q-r}} \sum_{j_1,\cdots,j_{q-r}} g_{i_1} \otimes \cdots \otimes g_{i_{q-r}} \otimes g_{j_1} \otimes \cdots \otimes g_{j_{q-r}} \\
& & \qquad \times f(a_1,\cdots,a_l, b_1,\cdots,b_{r-l}, i_1,\cdots,i_{q-r} )
 f(a_1,\cdots,a_l, b_1,\cdots,b_{r-l}, j_1,\cdots,j_{q-r} ) \\
&=& \sum_{b_1,\cdots,b_{r-l}} \sum_{i_1,\cdots,i_{q-r}} \sum_{j_1,\cdots,j_{q-r}}\prod_{u=1}^{r-l} \lambda^{-1/2}_{b_u} g_{b_1} \otimes \cdots \otimes g_{b_{r-l}} \otimes g_{i_1} \otimes \cdots \otimes g_{i_{q-r}} \otimes g_{j_1} \otimes \cdots \otimes g_{j_{q-r}} \\
& & \qquad \times \sum_{a_1,\cdots,a_l}  f(a_1,\cdots,a_l, b_1,\cdots,b_{r-l}, i_1,\cdots,i_{q-r} )
 f(a_1,\cdots,a_l, b_1,\cdots,b_{r-l}, j_1,\cdots,j_{q-r} ) .
\end{eqnarray*}
Consequently,
\begin{eqnarray*}
& &\| g^{(n)} \star_r^l g^{(n)} \|_{L^2}^2
= \sum_{b_1,\cdots,b_{r-l}} \sum_{i_1,\cdots,i_{q-r}} \sum_{j_1,\cdots,j_{q-r}} \prod_{u=1}^{r-l} \lambda^{-1}_{b_u}\\
& & \quad \times \left[ \sum_{a_1,\cdots,a_l}  f(a_1,\cdots,a_l, b_1,\cdots,b_{r-l}, i_1,\cdots,i_{q-r} )
 f(a_1,\cdots,a_l, b_1,\cdots,b_{r-l}, j_1,\cdots,j_{q-r} )  \right]^2 \\
& & \leq \frac{1}{\alpha^{r-l}} \| g^{(n)} \star_l^l g^{(n)} \|^2_{L^2} .
\end{eqnarray*}
Since $\alpha>0$, this relation yields the desired implication: if $\| g^{(n)} \star_l^l g^{(n)} \|_{L^2}\to0$, then $\| g^{(n)} \star_r^l g^{(n)} \|_{L^2}\to 0$.
\end{proof} \\
\end{subsection}

\begin{subsection}{Proofs of the main results}

\begin{proof}[Proof of Theorem \ref{thm_HS_Poisson}]

\medskip

We shall first prove the implication {\bf (1)} $\Rightarrow$ {\bf (2)} for a general $q\geq 1$. For every $\lambda >0$, let $P(\lambda)$ be a Poisson random variable with parameter $\lambda$. For every integer $k\geq 0$, we introduce the mapping
\[
\lambda \mapsto \widetilde{T}_k(\lambda) = \E[(P(\lambda) -\lambda)^k], \quad \lambda >0,
\]
so that, for instance, $\widetilde{T}_0(\lambda) =1$ and $\widetilde{T}_1(\lambda) =0$. It is well-known (see e.g. \cite[Proposition 3.3.4]{pt_book}) that the following recursive relation takes place: for every $k\geq 1$,
\[
\widetilde{T}_{k+1}(\lambda) = \lambda\sum_{j=0}^{k-1}\binom{k}{j}\widetilde{T}_j(\lambda).
\]
Elementary considerations now yield that, for every $k\geq 1$, the mapping $\widetilde{T}_k(\cdot)$ is a polynomial of degree $(k-1)/2$ if $k$ is odd, and of degree $k/2$ is $k$ is even. As a consequence, for every real $q\geq 1$ the mapping
\[
\lambda \mapsto   \frac{\E[|P(\lambda) -\lambda|^q]}{\lambda^{q/2}}
\]
is bounded on the set $[\alpha, \infty)$. Using (\ref{e:pi}) together with the assumption $\alpha = \inf_{i}\lambda_ i>0$, we infer that $\sup_{i\geq 1} \E[|P_i|^q] <\infty$ for every $q\geq 1$. Standard hypercontractivity estimates (see for instance \cite[Lemma 4.2]{npr2}) yield therefore that, since $\E[(F^{(n)})^2] \to \sigma^2$, then $\sup_{n\geq 1} \E[|F^{(n)}|^q] <\infty$, for every $q\geq1$.  As a consequence, if {\bf (1)} is in order, then necessarily $\E[(F^{(n)})^k] \to \E[Y^k]$ for every integer $k\geq 1$; in particular, {\bf (2)} is verified.

\medskip

\noindent Now assume that $q=1$ and {\bf (2)} is verified. A quick computation reveals that
\[
\E[(F^{(n)})^4]  - 3\E[(F^{(n)})^2]^2 =\|g^{(n)}\|^4_{L^4} =  \sum_{i=1}^{N^{(n)}} f_n(i)^4\frac{1}{\lambda_i} ,
\]
thus yielding the implication {\bf (2)} $\Rightarrow$ {\bf (3a)}. On the other hand, if {\bf (3a)} is verified, then one has that (by the Cauchy-Schwarz inequality)
\[
\|g^{(n)}\|^3_{L^3} \leq \|g^{(n)}\|_{L^2}\|g^{(n)}\|^{1/2}_{L^4} \to 0, \quad \text{ as }\,\, n\to\infty,
\]
so that the implication {\bf (3a)} $\Rightarrow$ {\bf (1)} follows from \cite[Corollary 3.4]{pstu}, thus concluding the proof for $q=1$.

\medskip

\noindent We now fix $q\geq 2$.
The implication {\bf (3b)} $\Rightarrow$ {\bf (1)} is a direct consequence of \cite[Theorem 5.1]{pstu}, whereas the equivalence between {\bf (3b)} and {\bf (3b')} follows from Proposition \ref{prop_ineq_star}. In view of the first part of the proof, we need only to show that {\bf (2)} $\Rightarrow$ {\bf (3b')}.
We start by exploiting formula (\ref{product2}) in order to write the chaotic decomposition of $I_q(h^{(n)})$, namely:
$$ I_q(g^{(n)})^2 = \sum_{k=0}^{2q} I_k\big( G_k^{q,q}(g^{(n)},g^{(n)}) \big) .$$
As a consequence, by exploiting the orthogonality of multiple integrals with different orders,
\begin{eqnarray}
\E[I_q(g^{(n)})^4] &=& \sum_{k=0}^{2q} k! \| G_k^{q,q}(g^{(n)},g^{(n)}) \|_{L^2}^2 \nonumber\\
 &=& \| G_0^{q,q}(g^{(n)},g^{(n)}) \|_{L^2}^2 + (2q)! \| G_{2q}^{q,q}(g^{(n)},g^{(n)}) \|_{L^2}^2\label{moment4_expression}\\
 &&\qquad\qquad\qquad\qquad\qquad\qquad+ \sum_{k=1}^{2q-1} k! \| G_k^{q,q}(g^{(n)},g^{(n)}) \|_{L^2}^2 , \notag
\end{eqnarray}
where
$$ \| G_0^{q,q}(g^{(n)},g^{(n)}) \|_{L^2}^2 = q!^2 \|g^{(n)} \|_{L^2}^4 ,$$
and
\begin{eqnarray} \label{expression_G_2q_q}
 (2q)! \| G_{2q}^{q,q}(g^{(n)},g^{(n)}) \|_{L^2}^2 &=& (2q)! \| \widetilde{g^{(n)}\star_0^0 g^{(n)}} \|_{L^2}^2\\
 & =& 2q!^2 \|g^{(n)} \|^4 + \sum_{p=1}^{q-1} \frac{(q!)^4}{(p!(q-p)!)^2} \| g^{(n)} \star_p^p g^{(n)} \|_{L^2}^2 ,\notag
\end{eqnarray}
where we have used \cite[formula (11.6.30)]{pt_book}. Since $q!^2 \|g^{(n)} \|_{L^2}^4\to\sigma^4$ by assumption, we deduce that, if {\bf (2)} is verified, then $\| g^{(n)} \star_p^p g^{(n)} \|_{L^2}\to 0$ for every $p=1,\ldots,q-1$, and the desired implication follows from Proposition \ref{prop_ineq_star}.
%
%
\end{proof}

\bigskip


\begin{proof}[Proof of Theorem \ref{thm_universality_Poisson_1}]

\noindent By virtue of Theorem \ref{t:unigauss}, it suffices to show that, if condition {\bf (1)} in Theorem \ref{thm_universality_Poisson_1} is in order, then the sequence $\{Q_q(N^{(n)}, f^{(n)}, {\bf G}) : n\geq 1\}$ converges in distribution to $Y$. Using the same notation as in Example \ref{example_Gauss}, with $\mathfrak{H} = L^2(\mu)$ and $h_i = g_i$, one has that the homogeneous sum $Q_q(N^{(n)},f^{(n)}, {\bf G})$ can be represented as a multiple Wiener-It\^o integral as follows:
$$Q_q(N^{(n)},f^{(n)}, {\bf G}) = \sum_{i_1, \cdots, i_q}^{N^{(n)}} f^{(n)}(i_1, \cdots, i_q) G_{i_1} \cdots G_{i_q} = I_q^{G}(h^{(n)}),$$
where
$$h^{(n)} = g^{(n)} = \sum_{i_1,\cdots,i_q}^{N^{(n)}} f^{(n)}(i_1,\cdots,i_q) g_{i_1} \otimes \cdots \otimes g_{i_q}.$$
Now, if condition {\bf (1)} in Theorem \ref{thm_universality_Poisson_1} holds, then for every $r=1,\ldots,q-1$, $\|g^{(n)}\star^r_r g^{(n)}\|_{L^2} \to 0$, and we immediately deduce the conclusion by combining Theorem \ref{thm_HS_Poisson} and  Theorem \ref{thm_equiv_Gauss_copy}. \end{proof}

\end{subsection}
\end{section}

\bibliographystyle{plain}

\begin{thebibliography}{}

\end{thebibliography}


\begin{thebibliography}{10}

\bibitem{DFR} L. Decreusefond, E. Ferraz and H. Randriam (2011). Simplicial homology of random configurations. Preprint.

\bibitem{deJongBook} P. de Jong (1989). Central Limit Theorems for generalized multilinear forms. Tract {\bf 61}. CWI, Amsterdam.

\bibitem{deJongMulti} P. de Jong (1990). A central limit theorem for generalized multilinear forms. \textit{J. Mult. Anal.} \textbf{34}, 275-289.

\bibitem{Dudley book} R.M.\ Dudley (2003). \textit{Real Analysis and
Probability }(2$^{\text{nd}}$ Edition). Cambridge University
Press, Cambridge.

\bibitem{FV} E. Ferraz and A. Vergne (2011). Statistics of geometric random simplicial complexes. Preprint.

\bibitem{kab}
Y. Kabanov (1975).
\newblock On extended stochastic integrals.
\newblock {\em Theor. Probab. Appl}, 20: 710--722.

\bibitem{ledoux_talagrand_book}
M. Ledoux and M. Talagrand (1990).
\newblock {\em Probability on {B}anach spaces}.
\newblock Springer-Verlag, Berlin Heidelberg New York.

\bibitem{MOO} E. Mossel, R. O'Donnell and K. Oleszkiewicz (2010). Noise stability of functions with low influences: variance and optimality. {\it Ann. Math.} {\bf 171}, 295-341.

\bibitem{np}
I.~Nourdin and G.~Peccati (2009).
\newblock {S}tein's method on {W}iener chaos.
\newblock {\em Prob. Theory Related Fields}, 145(1-2): 75--118.

\bibitem{nourdin_peccati_book}
I.~Nourdin and G.~Peccati (2012).
\newblock {\em Normal Approximations Using {M}alliavin Calculus: from {S}tein's
  Method to {U}niversality}.
\newblock Cambridge Tracts in Mathematics. Cambridge University Press.



\bibitem{NouPecMat} I. Nourdin and G. Peccati (2010). Universal Gaussian fluctuations of non-Hermitian matrix ensembles: from weak convergence to almost sure CLTs.  {\em ALEA} 7: 341-375.

\bibitem{npr4}
I.~Nourdin, G.~Peccati, and G.~Reinert (2010).
\newblock Stein's method and stochastic analysis of {R}ademacher sequences.
\newblock {\em Electron. J. Probab.}, 15: 1703--1742.

\bibitem{npr2}
I.~Nourdin, G.~Peccati and G.~Reinert (2010).
\newblock Invariance principles for homogeneous sums: Universality of
  {G}aussian {W}iener chaos.
\newblock {\em Ann. Probab.}, 38(5): 1947-1985.


\bibitem{npr}
I.~Nourdin, G.~Peccati, and A.~R\'eveillac (2010).
\newblock Multivariate normal approximation using {S}tein's method and
  {M}alliavin calculus.
\newblock {\em Ann. Inst. H. Poincar\'{e} Probab. Statist}, 46(1): 45-58.

\bibitem{nualart}
D.~Nualart (2006).
\newblock {\em The Malliavin calculus and related topics (2nd edition)}.
\newblock Springer-Verlag, Berlin.

\bibitem{no}
D.~Nualart and S.~Ortiz-Latorre (2008).
\newblock Central limit theorems for sequences of multiple stochastic
  integrals.
\newblock {\em Stochastic Process. Appl.}, 118(4): 614--628.

\bibitem{nuapec}
D.~Nualart and G.~Peccati (2005).
\newblock Central limit theorems for sequences of multiple stochastic
  integrals.
\newblock {\em Ann. Probab.}, 33(1): 177--193.

\bibitem{nuaviv}
D.~Nualart and J.~Vives (1990).
\newblock Anticipative calculus for the {P}oisson process based on the {F}ock  space.
\newblock In {\em Sem. de Proba. XXIV}, pages 154--165. Springer, Berlin,
\newblock LNM 1426.

\bibitem{PLR1} G. Peccati and R. Lachi\`eze-Rey (2011). Fine Gaussian fluctuations on the Poisson space: cumulants contractions and random graphs. In preparation.

\bibitem{pstu}
G.~Peccati, J.L. Sol\'e, M.S. Taqqu, and F.~Utzet (2010).
\newblock {S}tein's method and normal approximation of {P}oisson functionals.
\newblock {\em Ann. Probab.}, 38(2): 443--478.

\bibitem{pt}
G.~Peccati and M.S. Taqqu (2008).
\newblock Central limit theorems for double {P}oisson integrals.
\newblock {\em Bernoulli}, 14(3): 791--821.

\bibitem{pt_book}
G.~Peccati and M.S. Taqqu (2011).
\newblock {\em Wiener Chaos: Moments, Cumulants and Diagrams: A survey with
  Computer Implementation}.
\newblock Bocconi University Press and Springer, Milan.

\bibitem{ptudor}
G.~Peccati and C.A. Tudor (2005).
\newblock Gaussian limits for vector-valued multiple stochastic integrals.
\newblock In {\em S\'eminaire de Probabilit\'es XXXVIII}, pages 247--262.


\bibitem{peczheng}
G.~Peccati and C.~Zheng (2010).
\newblock Multi-dimensional Gaussian fluctuations on the Poisson space.
\newblock {\em Electron. J. Probab.}, 15: 1487--1527.

\bibitem{priv}
N.~Privault (2009).
\newblock {\em Stochastic Analysis in Discrete and Continuous Settings: With
  {N}ormal Martingales}.
\newblock Springer, Berlin.

\bibitem{lesmathias} M. Reitzner and M. Schulte (2011). Central Limit Theorems for U-Statistics of Poisson Point Processes. Preprint.

\bibitem{SchTh} M. Schulte and C. Thaele (2010). Exact and asymptotic results for intrinsic volumes of Poisson k-flat processes. Preprint.

\bibitem{surg2}
D.~Surgailis (1984).
\newblock On multiple {P}oisson stochastic integrals and associated {M}arkov
  semigroups.
\newblock {\em Probab. Math. Stat.}, 3(2): 217--239.

\end{thebibliography}

\end{document}